[1994/12/01]
\documentclass[a4paper]{aomart}

\usepackage{amsmath,amssymb}
\usepackage{xcolor}
\usepackage{enumitem}

\chardef\bslash=`\\ 

\newtheorem[{}\it]{thm}{Theorem}[section]
\newtheorem{cor}[thm]{Corollary}
\newtheorem{lem}[thm]{Lemma}
\newtheorem{prop}[thm]{Proposition}

\theoremstyle{definition}
\newtheorem{defn}{Definition}[section]
\newtheorem{rem}{Remark}[section]
\newtheorem*[{}\it]{notation}{Notation}

\newcommand{\N}{\mathbb{N}}

\newcommand{\R}{\mathbb{R}}

\newcommand{\Om}{\Omega}

\newcommand{\ric}{\mathrm{Ric}}

\renewcommand{\div}{\mathrm{div}}
\newcommand{\roll}{\mathrm{roll}}
\newcommand{\hess}{\mathrm{Hess~}}

\newcommand{\n}{\mathbf{n}}
\newcommand{\grad}{\nabla}
\newcommand{\dist}{\mathrm{dist}}
\newcommand{\seg}{\mathrm{seg}}
\newcommand{\inj}{\mathrm{inj}}
\newcommand{\cut}{\mathrm{cut}}
\newcommand{\T}{\mathrm{T}}

\newcommand{\spn}{\mathrm{span}}
\newcommand{\tr}{\mathrm{Tr}}
\newcommand{\h}{\tilde{\mathfrak{h}}}
\newcommand{\arcot}{\mathrm{arcotan}}

\newcommand{\eval}[2][\right]{\relax
  \ifx#1\right\relax \left.\fi#2#1\rvert}

\title[Eigenvalues of the Wentzel-Laplace operator]{ Wentzel-Laplace eigenvalues comparison}

\author{Aissatou M. NDIAYE}
\address{Institut de math\'ematiques\\
Université de Neuchâtel\\
Neuchâtel, Switzerland}
\fulladdress{Rue Emile-Argan, 11\\
Neuchâtel, Switzerland}
\email{aissatou.ndiaye@unine.ch}
\givenname{A\"issatou}
\surname{Ndiaye}


\startpage{1}
\endpage{}

\begin{document}

\begin{abstract}
In this paper we present quantitative comparisons  between the Wentzel-Laplace eigenvalues, Steklov eigenvalues and Laplacian eigenvalues on the boundary of the target manifold using Riccati comparison techniques to estimate the Hessian of the distance function from the boundary.
\end{abstract}

\maketitle
\tableofcontents

\section{Introduction}
Let $n\geqslant 2$ and $ (M, g)$ be an $n$-dimensional compact Riemannian manifold with  smooth boundary $\Gamma$. 
Let $\Delta$ and  $\Delta_\Gamma$ denote the Laplace-Beltrami operators  acting on functions  on $M$ and $\Gamma$ respectively. We define the Laplacian as the negative divergence of the gradient operator. The gradient operators on $M$ and $\Gamma$ will be denoted by $\nabla$ and $\nabla_\Gamma$ respectively and the outer normal derivative on $\Gamma$ by $\partial_\n$.  Throughout the paper we denote by $\mathrm{d}_M$ and $\mathrm{d}_\Gamma$ the (Riemannian) volume elements of $M$ and $\Gamma$.
Let $\beta\in\R_{\geqslant 0}$, we consider the Wentzel eigenvalue problem on $M$:
\begin{equation}\label{w}
\begin{cases}
\Delta u=0 \quad \text{in}~\Omega,\\
\beta \Delta_\Gamma u +\partial_\n u=\lambda u  \quad \text{on}~\Gamma.
\end{cases}
\end{equation}

Problem \eqref{w} admits a discrete sequence of eigenvalues that can be arranged as 
 \begin{equation}\label{spectrum}
 0=\lambda_{W,0}^{\beta}< \lambda_{W,1}^{\beta}\leqslant\lambda_{W,2}^{\beta}\leqslant \cdots\leqslant \lambda_{W,k}^{\beta}\leqslant \cdots \nearrow \infty.
 \end{equation}
We adopt the convention that each eigenvalue is repeated according to its multiplicity.


The eigenvalue problem of the Laplacian with Wentzel boundary has only recently been significantly investigated. There have also been new developments on the Steklov eigenvalue problem. See for example \cite{provstub,xiong2017,colbGirHas,xia2019escobars}. We adopt the philosophy of [Gal15], to interpret the Wentzel eigenvalue problem as a perturbed version (unperturbed when $\beta=0$) of the Steklov problem. This allows us to use similar methods as proposed in the recent works of Provenzano-Stubbe, Xiong and Colbois-Girouard-Hassannezhad. They use geometric properties of a well chosen distance function to bound Steklov eigenvalues. We, nevertheless, will focus on the Wentzel eigenvalues with boundary parameter $\beta>0$.

 Consider the map  $ \wedge: L^2 (\Gamma) \longrightarrow L^2 (\Omega) $  related to the  Dirichlet problem
\begin{equation}\label{harm}
\begin{cases}
\Delta u=0\quad\text{ in } \Omega,\\
u |_\Gamma= f \quad \text{ on } \Gamma,
\end{cases}
\end{equation}
which associates to any $ f \in  L^2 (\Gamma)$ its harmonic extension, that is, the unique function $u$ in $ L^2(\Omega)$ satisfying \eqref{harm}. This map is well defined from $L^2(\Gamma)$ (respectively, $H^{\frac{1}{2}}(\Gamma)$) to $L^2(\Omega)$ (respectively, $H^1(\Omega)$). 
See \cite[p. 320 Prop $1.7$]{Taylor} for more details. By $H^s(\Omega)$ and $H^s(\Gamma)$,  we denote the Sobolev spaces of order $s$ on $\Omega$ and $\Gamma$, and $u |_\Gamma\in H^\frac{1}{2}(\Gamma)$ stands for the trace of $u\in H^1(\Omega)$ at the boundary $\Gamma$. This will also be denoted by $u$,  if there is no ambiguity.

 Then the Dirichlet-to-Neumann operator is defined by 
\begin{align}
\mathrm{N_ D}: H^\frac{1}{2}(\Gamma) & \longrightarrow  H^{-\frac{1}{2}} (\Gamma)\\
 			f&\longmapsto \partial_\n (\wedge f)\nonumber.
\end{align}
Again $\partial_\n\in  H^{-\frac{1}{2}}(\Omega) $  stands for the  normal derivative at the boundary $\Gamma$ of $\Omega$ with $\n$ the normal vector pointing outwards.

For all $u\in H^{\frac{1}{2}}(\Gamma)$, we define the operator $\mathrm{B}_0= N_D$ (in the operator sense). For  $\beta>0$, we define $\mathrm{C}_\beta u \stackrel{\scriptscriptstyle\text{def}}= \beta\Delta_\Gamma u$ for all $u\in H^1(\Gamma)$ and $$\mathrm{B}_\beta \stackrel{\scriptscriptstyle\text{def}}= \mathrm{B}_0+\mathrm{C}_\beta.$$
The eigenvalues sequence $\{ \lambda_{W,k}^{\beta}\}_{k=0}^\infty $ given in \eqref{spectrum} can be interpreted as the spectrum associated to the operator $\mathrm{B}_\beta$ and  is subject to the following min-max characterisation (see e.g., \cite[Thm 1.2]{sandgren} and \cite[(2.33)]{Gal2015}).

Let $\mathfrak{V}(k) $ denote  the set of  all $ k$-dimensional  subspaces of   $ \mathfrak{V}_\beta$ which  is defined by
\begin{align}
 \mathfrak{V}_0&\stackrel{\scriptscriptstyle\text{def}}=\{(u,u_\Gamma)\in H^1(\Omega)\times H^\frac{1}{2}(\Gamma): u_\Gamma=u|_\Gamma \},\\
 \mathfrak{V}_\beta&\stackrel{\scriptscriptstyle\text{def}}=\{(u,u_\Gamma)\in H^1(\Omega)\times H^1(\Gamma): u_\Gamma=u|_\Gamma \}, \quad \text{for } \beta>0.
\end{align}
Of course, for all $\beta>0$, we have $ \mathfrak{V}_\beta\subset \mathfrak{V}_0$. For every $k\in\N$, the $k$th eigenvalue of the Wentzel-Laplace operator $B_\beta$ satisfies

\begin{equation}\label{char}
 \lambda_{W,k}^{\beta}={\underset{V\in \mathfrak{V}(k)}{\min}  }\underset {0\neq u\in V} {\max} R_\beta(u), \quad k\geqslant 0,
\end{equation}
where $R_\beta(u) $,  the Rayleigh quotient for $\mathrm{B}_\beta$, is given by 
\begin{equation}\label{rayleigh}
R_\beta(u) \stackrel{\scriptscriptstyle\text{def}}=\frac{\int_\Om{|\nabla u|^2 \mathrm{d}_M+\beta\int_{\Gamma}{|\nabla_\Gamma u|^2 \mathrm{d}_\Gamma}}}{\int_{\Gamma}{u^2 \mathrm{d}_\Gamma}}, \quad \text{for all } u\in \mathfrak{V}_\beta\backslash\{0\}.
\end{equation}

The eigenvalues for the Dirichlet-to-Neumann map $\mathrm{B}_0=\mathrm{N_ D}$ are those of the well-known Steklov problem:
\begin{equation}\label{Steklov}
\begin{cases}
\Delta u=0, & {\rm in \ }\Omega,\\
\partial_\n u=\lambda^S u, & {\rm on \ }\Gamma.
\end{cases}
\end{equation}
A good discussion of this problem can be found in \cite{GP}. The Steklov eigenvalues are then  $ \{ \lambda_{W,k}^{0}\}_{k=0}^\infty$
which we shall denote equivalently as $\{ \lambda_{S,k}\}_{k=0}^\infty $.
 They behave according to the following asymptotic formula:
 \begin{equation}
\label{WeylS}
\lambda_{S,k}= C_n k^{\frac{1}{n-1}}+O(k^{\frac{1}{n-1}}),\quad k\rightarrow\infty,
\end{equation}
where  $C_n=\frac{2\pi}{\left(\omega_{n-1} Vol(\Gamma)\right)^{\frac{1}{n-1}}} $. We refer the reader to \cite[Section 4]{sandgren}. For $\beta>0$, the Weyl asymptotic for $ \lambda_{W,k}^{\beta}$  can be deduced  directly from properties of perturbed forms using the   asymptotic behaviour of the spectrum of $C_\beta$,
\begin{equation}
\lambda_{C_\beta,k}=\beta C_n^2k^{\frac{2}{n-1}}+O(k^{\frac{2}{n-1}}),\quad k\rightarrow\infty,
\end{equation}
and thanks to H\"ormander:
 \begin{equation}
\label{WeylW}
\lambda_{W,k}^{\beta}=\beta C_n^2k^{\frac{2}{n-1}}+O(k^{\frac{2}{n-1}}),\quad k\rightarrow\infty.
\end{equation}
See in particular \cite[Prop 2.7, (2.37)]{Gal2015}.

 Let $n\geqslant 2$ and $ (M, g)$ be an $n$-dimensional compact connected Riemannian manifold with  smooth boundary $\Gamma$, we denote by $\{ \eta_{k}\}_{k=0}^\infty $ the eigenvalues of $\Delta_\Gamma$. From here on we always assume that $\beta>0$ is fixed. 

Let $n\in\N_{\geqslant 2}$, $\overline{h}\in\R_{>0}$ and $K_-,K_+,\kappa_-,\kappa_+\in \R $. Throughout the paper, we  designate by
\begin{enumerate}
\item
$\mathfrak{M}^n(K_-,\kappa_-)$  the class of all smooth compact Riemannian manifolds $M$ of
dimension $n$ with non-empty boundary such that
the Ricci curvature of $M$ is bounded from below by $(n-1)K_-$ and the mean curvature of the boundary by $\kappa_-$:
\begin{itemize}
\item $ (n-1)K_- \leqslant \ric \text{ in } M_{\overline{h}},$ 
\item  $ \kappa_-\leqslant H_0.$
\end{itemize}
\item $\mathfrak{M}^n(K_-,K_+,\kappa_-,\kappa_+)$ the class  of all smooth compact Riemannian manifolds $M$ of dimension $n$ and whose sectional curvature $K$ and principal curvatures $\{\kappa_i, i=1\ldots,n-1\}$ of the boundary $\Gamma $ are bounded in the following way:
\begin{itemize}
\item $K_-\leqslant K\leqslant K_+$ in $M_{\overline{h}},$
\item $\kappa_-\leqslant\kappa_1\leqslant\kappa_2\leqslant\ldots\leqslant\kappa_{n-1}\leqslant\kappa_+.$
\end{itemize}
\end{enumerate}
Here $\overline{h}$ stands for  the rolling radius of $M$. The definition is given in \eqref{ss19022020} and $M_{\overline{h}}:=\{x\in M, d_\Gamma(x)<\overline{h}\}$ designates the tubular neighbourhood of $\Gamma$ of width $\overline{h}$ where $ d_\Gamma$ is the distance function from the boundary $\Gamma$. 

Because of the characterisation \eqref{char} with the min-max principle for the eigenvalues $\{\eta_k \}_{k=0}^\infty$ and $\{\lambda^S_k\}_{k=0}^\infty$, we immediately have  for every $k\in \N$ that
$$\lambda^S_k+\beta\eta_k\leqslant \lambda_{W,k}^{\beta}. $$
In our results we give reciprocal comparisons.
\begin{thm}\label{mainth1}
There exists an explicit constant $\overline{B} $ such that on each manifold $M$ in the class $\mathfrak{M}^n(K_-,K_+,\kappa_-,\kappa_+)$ the following
inequality is satisfied 
\begin{equation}
\lambda_{W,k}^{\beta} \leqslant  \left[\frac{1}{\sqrt{\beta}}+\sqrt{\overline{B}+ \beta\eta_k}  \right]^2,\qquad \forall~k\in\N.
\end{equation}
\end{thm}
\begin{thm}\label{mainth2}There exists an explicit constant $\overline{A} $ such that on each manifold $M$ in the class $\mathfrak{M}^n(K_-,K_+,\kappa_-,\kappa_+)$ the following
inequality is satisfied 
\begin{equation}
 \lambda_{W,k}^{\beta}\leqslant(1+\beta\overline{A})\lambda^S_k+ \beta(\lambda^S_k)^2,\qquad \forall~k\in\N.
\end{equation}
\end{thm}

\begin{thm}\label{AvecReilly}
Let $\kappa_->0$, then each manifold in the class $\mathfrak{M}^n(K_-,\kappa_-)$ satisfies the following
inequality,
\begin{equation}
\lambda_{W,k}^{\beta} \leqslant  \frac{1}{2(n-1)\kappa_-}\left[\left(K_-+2\eta_k \right)+\sqrt{(K_-+2\eta_k)^2-4\kappa_-(n-1)\eta_k} \right]+\beta\eta_k.
\end{equation}
\end{thm}
 

To prove our main results, analytical properties of the distance functions served as a key tool.
This perspective is inspired by \cite{provstub},\cite{xiong2017} and \cite{colbGirHas}.
Let $(M,g)$ be a smooth compact $n$-dimensional Riemannian manifold with non-empty boundary $\Gamma$. We define the distance function from   $\Gamma$ by $d_\Gamma(x) = \inf\{\dist(x, y), y\in \Gamma \}$ for every $x$ in $M$, where $\dist$ is the distance in $M$ induced by the metric tensor $g$. The level sets of $d_\Gamma$ will be denoted by $\Gamma_r:=\{x\in M, d_\Gamma(x)=r \}$. Assume that $r>0$ is small enough such that $\Gamma_r=d_\Gamma^{-1}(r)$ is open and consists entirely of regular points for $d_\Gamma$. In this case $\Gamma_r$ is clearly a hypersurface.
The interesting fact about the distance function $d_\Gamma$ is that the inward normal vector field to $\Gamma_r$ is simply $\grad d_\Gamma$, and thus the second fundamental form  is the restriction of the Hessian $\hess(d_\Gamma)$ to $T\Gamma_r$; so the eigenvalues of  $\hess d_\Gamma$ are exactly the principal curvatures of $\Gamma_r$. The next section is devoted to sketching relevant properties of distance functions.
%


%
%

\section{Geometric properties of distance functions}
\subsection{Geometry of a distance function to a hypersurface}
 If $(M,g)$ is a Riemannian manifold and $\Sigma$ a compact smooth hypersurface in $M$, the  geometric quantities  that we are interested in here correspond to the eigenvalues of the Hessian of the (squared) distance function  from $\Sigma$, 
 $$d_\Sigma (x):=\dist(x,\Sigma)=\min\{\dist(x,y):~ y\in \Sigma \}$$
 for $x$ in a small tubular neighbourhood of $\Sigma$. These eigenvalues are related to the geometry of $\Sigma$ inasmuch as are exactly the principal curvatures of $\Sigma$. The restriction of the Hessian of $d_\Sigma$ to $T\Sigma$ coincides with the second fundamental form $\Pi_\Sigma$ of $\Sigma$. If $\Sigma$ is  compact of co-dimension one and  $2$-sided in $M$, to impose a unique correspondence one can constrain $d_\Sigma$ to be the signed distance function to $\Sigma$ given as the distance from $\Sigma$ with plus or minus sign depending on which side of $\Sigma$ we are on. However, since the signed distance function can make no sense, as in the case of higher co-dimension problems, a common element  to work with is the squared distance function
 $$\eta(x):=\frac{1}{2}d^2_{\Sigma}(x). $$ 
 See for instance \cite{AmbroisioMantegazza}. The factor $\frac{1}{2}$ is introduced  for convenience, to simplify several identities.

Our starting point, with the next proposition, is to set forth geometric concepts related to a specific smooth function, with the goal of arriving to curvatures on the Riemannian manifold. Let $(M,g)$ be a complete Riemannian manifold, $\Omega \subset M$ an open subset and  $f:\Omega \subset M\longrightarrow  \R$. Let  $\grad f$ denote the gradient and  $\hess f$ the Hessian of $f$. Let $\nabla$ denote the Levi-Civita connection and $S(X):= \grad_X \grad f$ be the $(1, 1)$-tensor corresponding to $\hess f$.
Let $a\in \R$ and $\Sigma \subset f^{-1}(a)$ open. We assume that $\Sigma$  consists entirely of regular points for $f$. In this case, $\Sigma$ is clearly a hypersurface. 
The second fundamental form of a hypersurface $\Sigma \subset M$ with a fixed unit normal vector field $\n: \Sigma \longrightarrow T^\perp =\{ v \in T_pM:~ p \in \Sigma, v \perp T_p\Sigma \}$ is defined as the $(0,2)$-tensor $\Pi(X,Y)= g(\grad_X N, Y)$ on $\Sigma$.
 Since $X,Y\in T\Sigma$ are perpendicular to $N$ we have $g (\grad_X N,N) = 0$.
One can also define $\Pi(X, \cdot) = g(\grad_X N, \cdot)=:g(\mathbf{S}(X),\cdot)$ for $X\in T\Sigma$ where $\mathbf{S}$ stands for the shape operator. 
We start by relating the second fundamental form of $\Sigma$ to $f$.
\begin{prop}\label{prop2101}
 The following properties hold:
 \begin{enumerate}
\item $ \frac{\grad f}{|\grad f|}$ is a unit normal to $\Sigma$,
\item $\Pi(X,Y)=\frac{1}{|\grad f|}\hess f (X,Y)$ for all $X, Y \in T\Sigma$.
 \end{enumerate}
\end{prop}
\begin{proof}
 \begin{enumerate}
\item If $X\in T\Sigma$ is a tangent vector field to $\Sigma$, $D_Xf=0$, then $ \frac{\grad f}{|\grad f|}$ is perpendicular to $\Sigma$. It is a unit normal to $\Sigma$ since it has unit length.
\item Take $N=\frac{\grad f}{|\grad f|}$ as unit normal from the previous point. Then
\begin{align*}
 g(\grad_X N, Y)&=g\left(\grad_X \left(\frac{\grad f}{|\grad f|}\right), Y\right)\\
 &=g\left(\frac{1}{|\grad f|}\grad_X \grad f, Y\right)+g\left(\grad_X \left(\frac{1}{|\grad f|}\right)\grad f, Y\right)\\
 &=\frac{1}{|\grad f|}g\left(\grad_X \grad f, Y\right).
\end{align*}
\end{enumerate}
\qedhere
\end{proof}

 In the same way, we now briefly recall some useful results connecting  analytical properties of distance functions and the geometric invariants of the manifold $M$.

Let $\Omega\subset (M,g)$ be an open domain.We call distance  function on $\Omega$ every smooth function  $f:\Omega\longrightarrow \R$ which is a solution to the Hamilton-Jacobi equation 
$$|\grad f|\equiv 1.$$
Level sets of a distance functions are hypersurfaces (i.e. $(n-1)$-dimensional submanifolds of $M$).

In what follows, we use the notation $\Gamma_s:=\{x\in \Omega, f(x)=s\}$ for the level sets of $f$. From Proposition \ref{prop2101}, one has on each level set, that  $\hess f=\Pi$, since $|\grad f|^2=1$. The $(1,1)$-tensor $S(X):=\grad_X\grad_f$  corresponds to both $\hess f$. It determines  how the unit normal to $\Gamma_s$ changes. The intuitive idea behind the operator  $S$, which stands here for the second derivative of $f$ but also for the shape operator, is to evaluate how the induced metric on $\Gamma_s$ changes.

Let $(M,g)$ be a smooth compact $n$-dimensional Riemannian manifold with non-empty boundary $\Gamma$. A particular distance function we are interested in is $d_\Gamma(x)= \inf\{ y\in \Gamma d(x, y)\}$, the distance function from the boundary, where $d$ is the distance coming from the metric tensor $g$. Let $s >0$ small enough such that $\Gamma_s:=d^{-1}_\Gamma(s)$ consists entirely of regular points for $d_\Gamma$. In this case, $\Gamma_s$ is the hypersurface bounding the submanifold $\Omega_s=\{x\in M : f(x)\geqslant s\}$.
Let $p\in \Gamma_s$ and $\T_p\Gamma_s$ denote the tangent space of $\Gamma_s$ at $p$. A given vector $v\in {\mathrm  {T}}_{p}M$ is normal to $\Gamma_s$ if $g(v,v')=0$  for all $v'\in {\mathrm  {T}}_{p}\Gamma_s$ (i.e. $v$ is orthogonal to ${\mathrm  {T}}_{p}(\Gamma_s))$. The normal space to $\Gamma_s$ at $p$ is  the space  ${\mathrm  {T}}_{p}^{\perp}(\Gamma_s)$ of all such $v$. From now on, we choose the unit normal vector field in such a way that the principal curvatures of the Euclidean sphere are taken to be positive. The following lemma links together the function $d_\Gamma$ and geometric data of $\Gamma_s$.
\begin{prop}
\label{prop321}
Let $X$ and $Y$ in $T\Gamma_s$. Then the following properties hold.
\begin{enumerate}[label=(\roman*)]
\item $|\grad d_\Gamma|=1$,
\item $\grad d_\Gamma$ is a unit normal to $\Gamma_s$,
\item $\Pi(X,Y)= \hess d_\Gamma(X,Y) $.\label{p3}
\end{enumerate}
\end{prop} 
The first point of the proof can be found in \cite[Thm 3.1]{MantegazzaMennucci2003}. From there, the rest of the proof is an immediate consequence of Proposition \ref{prop2101} after noting that  $|\grad d_\Gamma|=1$.

While dealing with derivatives, regularity matters are to be kept in mind.
In this section, we will present a summary of basic facts about distance functions and smoothness. We refer to Chavel's book \cite[Part \S III.2]{Chavel1993} for further details.
\subsection{Regularity of distance functions}\label{ss19022020}

Let $(M,g)$ be a complete Riemannian manifold and $p\in M$, we denote by $T_{p}M$ the tangent space at $p$.  It is known that the distance function from the point $p$, $\dist(p,\cdot)$ is smooth near $p$. Proposition \ref{Chardist} characterises the failure of $\dist(p,\cdot)$ to be smooth at some point, but before its statement, we need to introduce some necessary terminology.

Let $v \in T_pM$ and  $c_v:[0,1]\longrightarrow M$ the unique geodesic  satisfying  $c(0)=p$ and $\dot c(0)=v$. The exponential map from $p$ is defined by: $$\exp_p(v)=c_v(1).$$
 It is a standard result that for sufficiently small $v$ in $T_{p}M$, the curve defined by the exponential map, $ \gamma (t):[0,1]\ni t\longrightarrow\exp _{p}(tv)\in M$ is a segment (i.e. minimizing geodesic), and is the unique minimizing geodesic connecting the two endpoints.
 
 Let $\Omega_{p,q}:=\{c:[0,1]\longrightarrow M;~ c \text{ is piecewise } C^1 \text{ and } c(0)= p, c(1) =q\}$, for every $q\in M$.
 A curve $c\in \Omega_{p,q}$ is a segment if  it is parametrised by arc length, i.e.  $|\dot{c}|$ is constant and its length $L(c):=\int_0^1|\dot{c}|$ is equal to $\dist(p,q)$.
 The segment domain is defined by
$$\seg(p)=\{ v \in T_pM;~ \exp_p(tv): [0,1]\longrightarrow M \text{ is a segment}\}.$$
 We denote its interior by $\seg^0(p)$, 
$$\seg^0(p)=\{ tv;~ t\in [0,1],  v\in \seg(p)\}.$$
\begin{prop}\label{Chardist}
\begin{enumerate}
\item 
 Each element of $\exp_p\left(\seg^0(p)\right)$ is joined to $p$ by a unique segment, meaning that $\exp_p$ is injective on $\seg^0 (p)$.
 \item $\exp_p$ has no singularity in $\seg^0(p)$.
 \item 
 Every $v \in \seg(p) \backslash \seg^0(p)$, satisfies at least one of the two following conditions:
 \begin{itemize}
 \item  There exists $v'\in \seg(p)\backslash \{v\}$ such that $\exp_p(v')= \exp_p(v)$, or
 \item $D \exp_p$ is singular at $v$.
 \end{itemize}
\end{enumerate}
\end{prop}
 The cut locus of $p$ in the tangent space is defined to be the set $\seg(p) \backslash \seg^0(p$. 
The proposition states that it is the set of all $v \in T_{p}M$ such that the curve $\gamma: [0,1]\ni t\longmapsto\exp _{p}(tv)\in M$ is a segment   but becomes ineffectual whenever $t>1$.

Then it is natural to define the cut locus of $p$ in $M$ as the image of the cut locus of $p$ in the tangent space under the exponential map at $p$. This leads to the next corollary  that links the cut locus of $p$ in $M$,  the points in the manifold where the geodesics starting at $p$ stop being minimizing and the lack of smoothness of the distance function from $p$.
\begin{cor}
Let $c :[0,1]\longrightarrow M$ be a geodesic such that $c(0)= p$ and $\dot{c}(0)=v$. Let
$$\cut\left(v\right)= \sup\{ t;~ c\mid_{[0,t]} \text{is a segment}\},$$
then the distance function from $p$ satisfies the following
\begin{enumerate}
\item for every  $ t < \cut(v) $, $dist(p,\cdot)$ is smooth at $c(t)$,
\item $dist(p,\cdot)$ is not smooth at $ c(\cut(v))$.
\end{enumerate}
 In particular, irregularity of $\dist_p$ at some critical point $x$ is due to $\exp_p:\seg(p)\longrightarrow M$ either failing to be one-to-one at $x$ or having $x$ as a critical value.
 \end{cor}

The smallest distance from $p$ to the cut locus is called the injectivity radius of the exponential map at $p$, denoted $\inj_p$ and corresponds to the largest radius $r$ for
which $\exp_p:T_pM \supset B(0,r) \longrightarrow M$ is a diffeomorphism.

This leads to a natural  definition of cut locus for a submanifold $\Sigma$ of $M$ using the normal exponential map.
The normal exponential map is the map $\exp_\n :(0,\infty)\times \Sigma\ni(t,p)\longmapsto \exp_p t\n(p) \in M$, where $\n(p)$ is unit normal
vector to $\Sigma$ at $p$. 

Let $\n(\Sigma)$ denote the total space of normal bundles of $\Sigma$. If $\Sigma$ is a single point $p$ in $M$, then $\n(\Sigma)$ is the tangent vector space $T_pM$ of $M$ at $p$. If we denote $\n_t(\Sigma):=\{v\in\n(\Sigma): ||v||<t \}$ where $t>0$, then the normal exponential reads $\exp_\n :\n(\Sigma)\longrightarrow M$  and is nothing but the restriction of the exponential $\exp: TM\longrightarrow M$ to $\n(\Sigma)$.
 The global injectivity radius $\inj_p$ of $\exp_\n $ is defined to the supremum over all $t>0$ such that $\exp_\n \mid_{\n_t(\Sigma)}$ is an embedding.

 The injectivity radius of the distance to the boundary is often called the rolling radius of $M$. We denote it throughout the paper by $\roll(M)$. Point \ref{p3} of Proposition \ref{Chardist} emphasises that, at any $p\in \Gamma_r$, the eigenvalues of the Hessian of $ d_\Gamma$ $\hess d_\Gamma$ are
$$\{0,-\kappa_1,- \kappa_2, \ldots, -\kappa_n\}, $$
 where $\kappa_1, \kappa_2, \ldots, \kappa_n$ are the eigenvalues of the shape operator $S:\T\Gamma_r\ni X\longmapsto \grad_X\n\in \T\Gamma_r$
with  $\n=-\grad d_\Gamma$ corresponding to the outward unit normal vector field to $M_r$.
\subsection{Radial curvature equation} 
Let $(M,g)$ be a complete Riemannian manifold and $r$ a smooth distance function on an open subset of $M$. A Jacobi field for $r$ is a smooth vector field $J$ that does not depend on $r$. In other words, it satisfies the Jacobi equation
 \begin{equation}\label{Jacobi}
 [J,\grad r]=0.
 \end{equation}
 A particularly interesting fact about these Jacobi fields is that they can be used to compute the Hessian of the function $r$. 
 Let $J$ be a Jacobi field for $r$, then one has
 \begin{equation}\label{JacobieShapeOp}
 \grad_{\grad r}J=\grad_J \grad r=\grad_J  N=S(J).
 \end{equation}
 Let $R$ denote the Riemannian curvature tensor defined by 
 $$R(X,Y)Z=[\grad_X,\grad_Y]Z-\grad_{[X,Y]}Z$$
  for every vector fields $X$, $Y$ and $Z$.
From the above equalities \eqref{Jacobi} and \eqref{JacobieShapeOp} we get the following equations.
\begin{thm}
Let $s\in \R_{\geqslant 0}$ such that $\Gamma_s\subset r^{-1}(s)$ consists of regular points for $r$. Then we have:
 \begin{equation}
 \label{RiccatiS}
 \grad_{\grad f} S+S^2+R_{\grad f}=0
 \end{equation}and
 \begin{equation}\label{RiccatiH}
(\grad_{\grad_r}\hess r)(X,Y)+\hess^2r(X,Y)+R(X,\grad_r,\grad_r,Y)=0,
\end{equation}
 for every $X, Y$ in $T\Gamma_s$. Here $S^2$ corresponds to  $S\circ S$, $R_{\grad f}:=R(\cdot, \grad f)\grad f$ is the directional curvature operator (also called tidal force operator). The curvature tensor is changed  to a (0, 4)-tensor as follows:
$$R(X,Y,Z,W)= g(R(X,Y)Z,W).$$
\end{thm}
\begin{proof}
Notice that, $\grad_{\grad r}\grad r=0$ since
$$0=X(g(\grad r,\grad r))=2(g(\grad_X\grad r,\grad r)=2g(\grad_{\grad r}\grad r,X), \quad \forall X\in \T\Gamma_s.$$
One has 
\begin{equation*}
R(J,\grad r)\grad r=\grad_{J,\grad r}^2\grad r-\grad_{\grad r,J}^2\grad r.
\end{equation*}
Together with
\begin{equation*}
\grad_{J,\grad r}^2\grad r=\grad_J(\grad_{\grad r}\grad r)-\grad_{\grad_J\grad r}\grad r=-S\big(S(J)\big)
\end{equation*}
and 
\begin{equation*}
\grad_{\grad r,J}^2\grad r=\grad_{\grad r}(\grad_J\grad r)-\grad_{\grad_{\grad r}J}\grad r=\grad_{\grad r}\big(S(J)\big),
\end{equation*}
we get the first equality.
The second formula follows from the last point of Proposition \ref{prop2101}.
\end{proof}

\subsection{Application to the Rayleigh quotient for Wentzel eigenvalues}
Let $(M,g)$ be a smooth compact $n$-dimensional Riemannian manifold with non-empty boundary $\Gamma$. Let $d_\Gamma$ denote the distance function to $\Gamma$ as in Proposition \ref{prop321} and $\roll(M)$ the rolling radius of $M$, i.e the injectivity radius of $d_\Gamma$. Every $h\in(0,\roll(M))$, defines a so-called tubular neighbourhood of $\Gamma$:
 $$M_{h}:=\{p\in M: d_\Gamma(p)< h\}.$$
 It is the subset of $M$ bounded by $\Gamma$ itself and the level hypersurface $\Gamma_h:=\{p\in M: d_\Gamma(p)= h\}$. From $d_\Gamma$ one can define a distance function on $M_H$ from each level hypersurface $\Gamma_h$ by $$d_h(p):=h-d_\Gamma(p).$$
Consider the modified distance function $\eta:=\frac{1}{2}(d_h)^2$, we have the following decomposition of the Rayleigh quotient for harmonic functions of $M$. 
\begin{prop}\label{PohWentzel}
Let $u$ be a smooth harmonic function on $M$. Then, for every $h\in(0,\roll(M))$, one has
\begin{align}\nonumber
R_\beta(u)
&=\int_M{ |\grad u|^2 \mathrm{d}_M}\\
&\quad\quad-\frac{\beta}{h}\int_{M_h}{ |\grad u|^2\Delta \eta+2 \hess \eta(\grad u,\grad u)\mathrm{d}_M}\\
&\qquad\qquad+\beta\int_\Gamma (\partial_\n u)^2\mathrm{d}_\Gamma.
\end{align}
\end{prop}
\begin{proof}
$R_\beta(u)=\int_M{ |\grad u|^2 \mathrm{d}_M}+\beta \int_\Gamma { |\grad_\Gamma u|^2 \mathrm{d}_\Gamma}$.
Notice that, on $\Gamma$ one has, $|\grad u|^2=|\grad_\Gamma u|^2+(\partial_\n u)^2$ and $\grad\delta=\n$. Then applying the divergence theorem, we get
\begin{align*}
\int_\Gamma |\grad u|^2  \mathrm{d}_\Gamma&=\int_{M_h}{ \div(|\grad u|^2\grad\delta) \mathrm{d}_M}\\
&=\int_{M_h}{ D_{\grad \delta} (|\grad u|^2)+|\grad u|^2\div(\grad \delta) \mathrm{d}_M}\\
&=\int_{M_h}{ 2 g(\grad_{\grad \delta}\grad u,\grad u)+|\grad u|^2\div(\grad \delta) \mathrm{d}_M}
\end{align*}
and \begin{align}\label{eq1}
R_\beta(u)=&\int_M{ |\grad u|^2 \mathrm{d}_M}-\beta \int_\Gamma { (\partial_\n u)^2 \mathrm{d}_\Gamma}\nonumber\\
&+\frac{\beta}{h}\int_{M_h}{ 2 \hess u(\grad\delta,\grad u)+|\grad u|^2\div(\grad \delta) \mathrm{d}_M}.
\end{align}
Moreover, we have
 $$\hess u(\grad\delta,\grad u)=div(g(\grad\delta,\grad u)\grad u)-\hess \delta(\grad u,\grad u).$$ Indeed,
\begin{align*}
div(g(\grad\delta,\grad u)\grad u)&=D_{\grad u}g(\grad\delta,\grad u)+g(\grad\delta,\grad u)\div(\grad u)\\
&=g(\grad_{\grad u}\grad\delta,\grad u)+g(\grad\delta,\grad_{\grad u}\grad u)-g(\grad\delta,\grad u)\Delta u\\
&=\hess\delta(\grad u,\grad u)+\hess u(\grad \delta,\grad u).
\end{align*}
Hence, replacing in \eqref{eq1}, one has
\begin{align*}
R_\beta(u)
&=\int_M{ |\grad u|^2 \mathrm{d}_M}-\beta \int_\Gamma { (\partial_\n u)^2 \mathrm{d}_\Gamma}\\
&+\frac{\beta}{h}\int_{M_h}{ 2 \left[ div(g(\grad\delta,\grad u)\grad u)-\hess \delta(\grad u,\grad u)\right]+|\grad u|^2\div(\grad \delta) \mathrm{d}_M}\\
&=\int_M{ |\grad u|^2 \mathrm{d}_M}+\beta\int_\Gamma (\partial_\n u)^2  \mathrm{d}_\Gamma\\
&\quad-\frac{\beta}{h}\int_{M_h}{ |\grad u|^2\Delta \delta+2 \hess \delta(\grad u,\grad u)  \mathrm{d}_M},
\end{align*}
where the last line is obtained after integrating by part and combining terms in $\int_\Gamma { (\partial_\n u)^2 \mathrm{d}_\Gamma}$.
\qedhere
\end{proof}

\section{ Comparison estimates}
As already said, the estimate in our main theorem follows from a geometric Hessian comparison theorem. The idea is to compare a geometric quantity on a Riemannian manifold with the corresponding quantity on a model space. 
%
%
%
%
 \subsection{Constant curvature specifics}
 If $S$ is a surface with constant mean curvature $K$, then 
 \begin{enumerate}
 \item $S$ is isometric to a sphere of radius $a$  if $K=\frac{1}{a^2}$,
 \item $S$ is isometric to a plane if $K=0$,
 \item $S$ is isometric to a pseudo-sphere determined by $a$, if $K=-\frac{1}{a^2}$.
\end{enumerate}  
 We denote the hyperbolic space of curvature $\kappa<0$, Euclidean space $(\kappa=0)$, or the sphere of curvature $\kappa>0$ either  jointly  by $M_k^m$, or, if the sign of $\kappa$ is specified, by $\mathbb{H}^n_\kappa$, $\R^n$, $\mathbb{S}_k^n$:
\begin{defn}[Model spaces.]
Consider $\R^n$ endowed  with the euclidean metric. For any $R>0$,
$\mathbb{S}^n(R) := \{p\in \R^{n+1},~ |p| = R\}$ denotes the metric sphere of radius $R$ endowed with the induced  Euclidean metric from $R^{n+1}$. As well,
denote by $\mathbb{H}^n(R):= \{(p_0, \ldots, p_n) \in\R^{n+1},~ -p_0^2+ p_1^2+ \ldots + p_n^2 =-R^2\}$ the hyperbolic space with the restriction of the Minkowski metric of $\R^{n+1}$.
A $n$-dimensional model space $M_\kappa^n$ is a Riemannian manifold   with constant curvature $\kappa$, for some $\kappa\in\R$. We think of $M_\kappa^n$ as
\begin{equation*}
M_\kappa^n:=
\begin{cases}
\mathbb{S}^n_{\frac{1}{\sqrt{\kappa}}}:= \mathbb{S}^n(\frac{1}{\sqrt{\kappa}})\quad & \text{if }\kappa>0,\\
\R^n\quad &  \text{if }\kappa=0,\\
\mathbb{H}_{\frac{1}{\sqrt{-\kappa}}}^n:=\mathbb{H}^n(\kappa)\quad & \text{if }\kappa<0,
\end{cases}
\end{equation*}
since any complete simply connected $n$-dimensional Riemannian manifold of constant curvature  is isometric to one of the above model spaces.
\end{defn}

Let $\sigma$ be a compact smooth hypersurface of a model space $M$, a point $p$ of $\Sigma$ is umbilical if the principal
curvatures of $\Sigma$ at $p$ are equal and $\Sigma$ is called umbilical if every point is umbilical.
Assuming that $M=M_\kappa^m$ is the model space of constant sectional curvature $K$ with umbilical boundary having principal curvature $\kappa$, the shape operator of a parallel  hypersurfaces
 satisfies
$$
\begin{cases}
a'+a^2+K=0\\
a(0)=-\kappa.
\end{cases}
$$ 
The idea is to examine the situation in $M$ comparing with the case of described model space.

\subsection{ Riccati comparison}\label{SectionRiccatiComp}
We start with a general result for differential inequalities.
In the last section we discussed the Riccati equation as an equation of a field of endomorphisms on $T (c^\perp)$ along a curve $c$. Now we discuss the corresponding one-dimensional ODE of the same type.
\begin{defn}
 Let $\kappa\in\R$ and let $u : I\longrightarrow\R$ be a smooth function on the interval $I \subset \R$.
  Then $u$ is a solution of 
 \begin{enumerate}
 \item the Riccati inequality if
$$u'+ u^2\leqslant -\kappa.$$
\item the Riccati equation if
$$u' +u^2+\kappa=0,$$
\item the Jacobi equation if
$$u''+ \kappa u = 0.$$
 \end{enumerate}
It is a maximal solution to these differential equations if it is a solution defined on the interval $I$ such that there is no solution defined on a suitable interval $I'$, which properly contains $I$. 
\end{defn}
As (in)equations of vector fields, the study of this (in)equations  is a classic topic in
differential geometry.
For our comparison theory we require specific solutions to the Jacobi equation, which will be usable
all  over  the text. We also gather some useful properties.
\begin{defn}
We define $sn_\kappa$ as the unique solution of the Jacobi equation satisfying
$$sn_\kappa(0) = 0, \qquad sn'_\kappa(0)=1$$
and by $cs_\kappa$ as the unique solution of the Jacobi equation satisfying
$$cs_\kappa(0) = 1,\qquad cs'_\kappa(0) = 0.$$
\end{defn}
\begin{lem}[Properties of $sn_\kappa$ and $cs_\kappa$.] \label{Proprties2701}
For any $\kappa\in\R$, the following holds.
\begin{enumerate}
\item \label{item27011} The solutions are explicitly given by $sn_\kappa, cs_\kappa :\R\longrightarrow\R$
\begin{equation*}
sn_\kappa(t)=
\begin{cases}
\frac{1}{\sqrt{\kappa}}\sin(\sqrt{\kappa}t)\quad &  \kappa>0\\
t \quad &  \kappa=0\\
\frac{1}{\sqrt{-\kappa}}\sinh(\sqrt{-\kappa}t)\quad &  \kappa<0
\end{cases}
\quad
cs_\kappa(t)=
\begin{cases}
\cos(\sqrt{\kappa}t)\quad & if \kappa>0\\
1 \quad & if \kappa=0\\
\cosh(\sqrt{-\kappa}t)\quad & if \kappa<0.
\end{cases}
\end{equation*}

\item\label{item27012} Define
\begin{equation*}
R_\kappa:=
\begin{cases}
\frac{\pi}{\sqrt{\kappa}}\quad &  \kappa>0\\
\infty \quad &  \kappa\leqslant 0
\end{cases}
\quad \text{and} \quad
L_\kappa:=
\begin{cases}
\frac{\pi}{2\sqrt{\kappa}}\quad &  \kappa>0\\
\infty \quad &  \kappa\leqslant 0.
\end{cases}
\end{equation*}
Then, we have
$$sn_\kappa(t)>0 \text{ for all } t \in (0,R_\kappa) \quad,\quad  cs_\kappa(t) > 0 \text{ for all } t \in (0,L_\kappa ).$$
\item \label{item27013} for every $t\in \R$
$$ sn_\kappa(-t) = -sn_\kappa(t)\quad\text{ and }\quad cs_\kappa(-t) = cs_\kappa(t),$$
$$ 1=cs_\kappa^2(t)+\kappa sn_\kappa^2(t).$$
\item \label{item27014} These functions satisfy
$$sn_\kappa'=cs_\kappa\quad\text{ and }\quad cs'_\kappa= -\kappa sn_\kappa.$$
\end{enumerate}
\end{lem}
\begin{proof}
\eqref{item27011} It suffices to check that these functions solve the respective initial value problems.
The points \eqref{item27012} and \eqref{item27013} follow immediately from \eqref{item27011}.
For \eqref{item27014}, by computation we have
\begin{equation*}
\begin{cases}
(sn'_\kappa)''+\kappa sn'_\kappa=(sn''_\kappa)'+\kappa sn'_\kappa=0\\
sn'_\kappa(0)=1=cs_\kappa(0)\\
sn''_\kappa(0)=-\kappa sn_\kappa(0)=0=cs'_{\kappa}(0),
\end{cases}
\end{equation*}
and
\begin{equation*}
\begin{cases}
(cs'_\kappa)''+\kappa cs'_\kappa=(sn''_\kappa)'+\kappa cs'_\kappa=0\\
cs'_\kappa(0)=0=-\kappa sn_\kappa(0)\\
cs''_\kappa(0)=-\kappa cs_\kappa(0)=-\kappa=-\kappa sn'_{\kappa}(0).
\end{cases}
\end{equation*}
The result follows from the uniqueness of solutions of initial value problems.
\qedhere
\end{proof}

\begin{prop}[Riccati comparison principle ]\label{riccacomp}

Let $\rho_1$ and $\rho_2$ be two smooth functions $\rho_{1,2}:(0,b)\longrightarrow \R$ such that $\rho_i(t)=\frac{1}{t}+O(t)$,   $i=1,2$.
If $\rho_1$ and $\rho_2$ satisfy
$$\dot{\rho}_1+\rho^2_1 \leqslant\dot{\rho}_2+\rho^2_2,$$
then
$$\rho_1\leqslant\rho_2,$$
(cf. \cite[(1.6.1)]{Karcher1989}). We note $f(x)=O(g(x)){\text{ as }}x\to 0$
if and only if there exist positive numbers $\delta$ and $M$ such that
$$|f(x)|\leqslant Mg(x){\text{ when }}0<x<\delta.$$
\end{prop}
\begin{proof}
Define $\Phi_i(t)=t\exp(\int_0^t (\rho_i(s)-1/s)ds)$ on $[0,b)$. Then   $\Phi_i(0)=0$ and  $\frac{d}{dt}\left[ t\exp(\int_0^t (\rho_i(s)-1/s)ds)\right]$  exists at zero, $\Phi_i'(0)=1$. On $(0,b)$, $\Phi_i$ is smooth and strictly positive, differentiating gives 
\begin{align*}
\Phi'_i(t)&=\frac{d}{dt}\left[ t\exp(\int_0^t (\rho_i(s)-1/s)ds)\right] \\
	&=\frac{\Phi_i(t)}{t}+t\frac{d}{dt}\left[ \exp(\int_0^t (\rho_i(s)-1/s)ds)\right] \\
	&=\frac{\Phi_i(t)}{t}+t\frac{d}{dt}\left[ \int_0^t (\rho_i(s)-1/s)ds\right]\frac{\Phi_i(t)}{t} \\
	&=\rho_i(t) \Phi_i(t),
\end{align*}
and hence $\Phi''_i(t)=(\rho'_i(t)+\rho_i(t)^2)\Phi_i(t)$. It follows that 
$$\Phi_2'(t)\Phi_1(t)-\Phi_1'(t)\Phi_2(t)\geqslant\Phi_2'(0)\Phi_1(0)-\Phi_1'(0)\Phi_2(0)=0$$
 since the function $\Phi_2'(t)\Phi_1(t)-\Phi_1'(t)\Phi_2(t)$ is increasing:
$$(\Phi_2'\Phi_1-\Phi_1'\Phi_2)'=\Phi_2''\Phi_1-\Phi_1''\Phi_2=((\rho'_2+\rho_2^2)-(\rho'_1+\rho_1^2))\Phi_1\Phi_2\geq 0.$$
Computation of $\Phi_2'(t)\Phi_1(t)-\Phi_1'(t)\Phi_2(t)$ and the  previous inequality lead to $\rho_2(t)- \rho_1(t)\geq 0.$
\qedhere
\end{proof}

Here is the complete version of Proposition \ref{riccacomp}: 
\begin{thm}[Ricatti comparison principle]\label{RiccatiComparison}
For $i=1,2$, we  let $a_i:[0,m_{a_{i}})\longrightarrow \R$  be the maximal solution of the Riccati eqution $a_i'+a_i^2+R_i=0$.
$$\text{If }
 \begin{cases}
 a_1'+a_1^2\leqslant a_2'+a_2^2 ~(\text{i.e. }
 R_1\geqslant R_2) \\
 \text{ and }\\
a_1(0)\leqslant a_2(0)
 \end{cases}
\text{then }\quad
 \begin{cases}
 a_1(t)\leqslant a_2(t), t\in[0,m_{a_1})\\
  \text{ and }\\
 m_{a_{1}}\leqslant m_{a_{2}}.
 \end{cases}
$$
For a proof, we refer the reader to \cite[Prop 2.3]{Eschenburg1987}.
\end{thm}

\subsection{Sectionnal curvature comparison}
We let $n\in\N_{\geqslant 2}$, $\overline{h}\in\R_{>0}$ and $K_-,K_+,\kappa_-,\kappa_+\in \R $. For the reminder of this section, we assume that\\
 $M\in\mathfrak{M}^n(K_-,K_+,\kappa_-,\kappa_+) $ and $\overline{h}=\roll(M)$ is the rolling radius of $M$.

We define $M_h:=\{x\in M : d_\Gamma< h\}$ for every $h \in (0,\overline{h})$ and $d_h$ the distance function to the level set $\Gamma_h$:  
$$d_h : M_h\ni x \longmapsto h-d_{\Gamma}(x)\in\R_{\geqslant 0}.$$ 
Let $H(h)$ denote the mean curvature of $\Gamma_h$, then the following relation holds:
\begin{thm}[Mean curvature comparison]
\label{MeanCurvatureComp}
Let $\mu:[0,m_\mu)\longrightarrow\R$ be the maximal solutions of 
$$
\begin{cases}
\mu'+\mu^2+K_-=0\\
\mu(0)=-\kappa_-.
\end{cases}
$$
 Then
$$-\mu(h)\leqslant \mathrm{H}(h),\quad \forall~ h\in [0, \overline{h})$$ and  $$\roll(M)\leqslant m_\mu.$$ 
\end{thm}
\begin{proof}

We apply the trace operator to equality \eqref{RiccatiH} within $r=d_\Gamma$. Indeed, take an
orthonormal frame $E_i$ in $T\Gamma_s$ and set  $X= Y =E_i$. Summing over $i$, we get 
$$(\grad_{\grad_r}\hess r)(X,Y)+\hess^2r(X,Y)+R(X,\grad_r,\grad_r,Y)=0,$$
$$\sum_{i=1}^nR(E_i,\grad_r,\grad_r,E_i)=\ric(\grad_r,\grad_r).$$
$$\sum_{i=1}^n (\grad_{\grad_r}\hess r)(E_i,E_i)=\sum_{i=1}^n (\partial_r \hess r)(E_i,E_i)=\partial_r\Delta r=(n-1)\mathrm{H}(s)$$
$$\sum_{i=1}^n \hess^2r(E_1,E_2)= |\hess r|^2\geqslant \frac{1}{n-1}(\Delta r)^2, $$
by the Cauchy-Schwarz inequality. This leads to
\begin{equation*}
\partial_r\Delta r+  \frac{1}{n-1}(\Delta r)^2 +\ric(\grad r,\grad r)\leqslant 0,
\end{equation*}
and equivalently

\begin{equation*}
 -\mathrm{H}' + \big( -\mathrm{H}\big)^2 +\frac{1}{n-1}\ric(\grad r,\grad r)\leqslant 0.
\end{equation*}
Then, for every $h\in(0,\overline{h})$
\begin{equation*}
\begin{cases}
 &-\mathrm{H}'(h) +  (-\mathrm{H}(h))^2 \leqslant -K_-=\mu'+\mu^2 \text{ and }\\
 & -\mathrm{H}(0)\leqslant -\kappa_-.
\end{cases}
\end{equation*}
It follows from Theorem \ref{RiccatiComparison} that $-\mathrm{H}(h)\leqslant \mu(h)$ for every $h\in[0,\overline{h})$ and $\overline{h}\leqslant m_\mu$.
\end{proof}
 In order to state the next lemma  we need to define a new constant 
 $\h:=\h(K_+,\kappa_+)$ as 
 \begin{equation}
 \label{eq160202020b}
\h=
 \begin{cases}
 \frac{1}{\max\{0,\kappa_+\}}\quad &\text{if } K_+=0\\
 \frac{1}{\sqrt{|K_+|}}\arcot\Big(\frac{\kappa_+}{\sqrt{|K_+|}}\Big)  &\text{if } K_+\neq 0,
 \end{cases}
 \end{equation}
 where (division by zero occurs) we accept, by convention, that the value of $\frac{1}{0}$ is $+\infty$.
\begin{lem}\label{lem11022020a}
Let $h \in(0,\h)$ and $p\in M_h\backslash \Gamma$ such that $d_\Gamma(p)=s$. That is $p\in \Gamma_s$ and $s\in[0, h)$ since $p$ is taken in $M_h$. Set $t:=h-s$ and denote by $\{\kappa_1(t)\leqslant\kappa_2(t)\leqslant\ldots\leqslant \kappa_{n-1}(t) \}$ the principal curvatures on $\Gamma_s$. Then the eigenvalues of the Hessian $\hess \eta$ at $p$ are given by
\begin{equation}
\rho_i(t)=
\begin{cases}
t\kappa_i(t) \leqslant 1\quad \forall i\in\{1,\ldots,n-1\}\\
\rho_n= 1.
\end{cases}
\end{equation}
\end{lem}
\begin{proof}
We notice that, from the point \ref{p3} of Proposition \ref{prop321}, the eigenvalues of the Hessian of $d_h$ at $p\in \Gamma_s$, $\hess d_h(p)=-\hess d_\Gamma(p)$  are zero and the principal curvatures of the boundary,
$$\{0,\kappa_1, \kappa_2, \ldots, \kappa_{n-1}\}. $$
We have $\hess \eta=\grad d_h\otimes \grad d_h+d_h\hess d_h $ and since $\grad d_h$ is a normal field, $\grad d_h\otimes \grad d_h(X,Y)$ vanishes for any tangent vector fields $X$ and $Y$. Hence, using that $d_h(p)=h-s$, we get $(h-s)\kappa_i I= \hess \eta $ in $T_p\Gamma_s$, for each $i=1,\ldots,n-1$. This means that the eigenvalues of $ \hess \eta$ at $p$ are given by 
$$\rho_i=(h-s)\kappa_i,\qquad \forall~i=1,\ldots,n-1 \quad \text{and} \quad\rho_n=1.$$

Let $x\in\Gamma_h$ be the nearest point to $p$ and $\n:=\grad d_h$ denote the unit outer normal vector  to $\Gamma_h$ at $x$. Setting $t=h-s$, one  has $p=x+t\n$.

We set $\tilde{K}:=|K_+|$, from the Riccati comparison principle (Theorem \ref{RiccatiComparison}), if  $a:[0,m_a)\longrightarrow\R$ is the maximal solution of 
\begin{equation}\label{eq22022020a}
\begin{cases}
a'+a^2+\tilde{K}=0\\
a(0)=-\kappa_+,
\end{cases}
\end{equation}
 then 
 \begin{equation}\label{eq19022020b}
 m_a\leqslant \roll(M) \text{ and }
 \end{equation}
 \begin{equation}\label{eq16022020a}
  a(t)\leqslant -\kappa_i\quad \text{for every } t\in[0, m_a).
 \end{equation}

 Now we have the following cases:
\begin{itemize}
\item[-] \underline{If $K_+=0$ (i.e. $\tilde{K}=0$)}
\begin{itemize}
\item[.]If $\kappa_+=0$, the constant function zero is the maximal solution of \eqref{eq22022020a} with maximal existence time $+\infty$.
\item[.]If $\kappa_+\neq 0$, we notice that $\mu(t)=\frac{cs_{{K}}}{sn_{{K}}}(t)$ satisfies
$$\mu'(t)+\mu^2(t)=-{K}$$
for every $K$ and $t$ in $\R$.

 Let $a(t):=\frac{cs_0}{sn_0}(t+t_0)$ with $t\in\R$ such that $a(0)=-\kappa_+$. That is 
$$\frac{1}{t_0}=-\kappa_+\Leftrightarrow t_0=-\frac{1}{\kappa_+}.$$
Then $$a(t)=\frac{cs_0}{sn_0}\left(t-\frac{1}{-\kappa_+}\right)=\frac{\kappa_+}{\kappa_+ t-1}$$ is the maximal solution of 
$$
\begin{cases}
a'+a^2+\tilde{K}=0\\
a(0)=-\kappa_+.
\end{cases}
$$
Since the only possible pole of $a$ is at
\begin{equation}\label{eq19022020b1}
 t=\frac{1}{\kappa_+},
\end{equation}
 replacing in \eqref{eq16022020a}, we get 
$$\kappa\leqslant \frac{\kappa_+}{1-t\kappa_+ } \text{ for every } t\in [0, m_{a}), $$
with 
$$m_a=
\begin{cases}
\frac{1}{\kappa_+}\quad &\text{ if } \kappa_+>0\\
+\infty&\text{ if } \kappa_+<0.
\end{cases}
$$
Hence for every $0\leqslant t<h\leqslant m_a$, we have
$$\rho_i=(h-t)\kappa_i\leqslant \left(\frac{1}{\kappa_+}-t\right)\frac{\kappa_+}{1-t\kappa_+ }= 1.$$
\end{itemize}
\item[-] \underline{If $K_+\neq 0$ (i.e. $\tilde{K}>0$)}
Notice that $\mu(t)=-K\frac{sn_{{K}}}{cs_{{K}}}(t)$ satisfies
$$\mu'(t)+\mu^2(t)=-{K} \text{ for every } K\neq 0 \text{ and } t \in \R.$$
Let $a(t):=-K\frac{sn_{\tilde{K}}}{cs_{\tilde{K}}}(t+t_0)$ with $t\in\R$ such that $a(0)=-\kappa_+$.\\
 That is 
$$-\sqrt{\tilde{K}}\tan(\sqrt{\tilde{K}}t_0)=-\kappa_+\Leftrightarrow t_0=\frac{1}{\sqrt{\tilde{K}}}\arctan\left(\frac{\kappa_+}{\sqrt{\tilde{K}}}\right).$$
Then 
\begin{align*}
a(t)&=-\sqrt{\tilde{K}}\tan\left(\sqrt{\tilde{K}}t+\arctan\left(\frac{\kappa_+}{\sqrt{\tilde{K}}}\right)\right)\\
&=-\sqrt{\tilde{K}}\cot\left(-\sqrt{\tilde{K}}t+\arcot\left(\frac{\kappa_+}{\sqrt{\tilde{K}}}\right)\right)
\end{align*}
is the maximal solution of 
$$
\begin{cases}
a'+a^2+\tilde{K}=0\\
a(0)=-\kappa_+.
\end{cases}
$$
Since the first pole of the function $\cot$ is $\frac{\pi}{2}$, the first pole of $a$ is 
\begin{equation}
\label{eq19022020b2}
m_a=\frac{1}{\sqrt{\tilde{K}}}\left[\frac{\pi}{2}-\arctan\left(\frac{\kappa_+}{\sqrt{\tilde{K}}}\right) \right]=\frac{1}{\sqrt{\tilde{K}}}\left[\arcot\left(\frac{\kappa_+}{\sqrt{\tilde{K}}}\right) \right].
\end{equation}
Replacing in \eqref{eq16022020a}, we get 

$$\kappa_i\leqslant \sqrt{\tilde{K}}\cot\left(-\sqrt{\tilde{K}}t+\arcot\left(\frac{\kappa_+}{\sqrt{\tilde{K}}}\right)\right). $$
Hence for every $0\leqslant t<h\leqslant m_a$, we have
\begin{align*}
\rho_i=(h-t)\kappa_i&\leqslant (m_a-t)\sqrt{\tilde{K}}\cot\left(-\sqrt{\tilde{K}}t+\arcot\left(\frac{\kappa_+}{\sqrt{\tilde{K}}}\right)\right)\\
&=\left((m_a-t)\sqrt{\tilde{K}}\right)\cot\left((m_a-t)\sqrt{\tilde{K}}\right)\\
 &\leqslant 1.
\end{align*}$$$$

\end{itemize}


\end{proof}
\begin{rem}
From inequality \eqref{eq19022020b} and explicit formulas in \eqref{eq19022020b1} and \eqref{eq19022020b2}, we see that $\h\leqslant \overline{h}$.
\end{rem}
We are now ready to begin proving Theorems \ref{mainth1} and \ref{mainth2}.
\subsubsection{First comparison}
In this section we prove Theorem \ref{mainth1}. We use the following technical lemmas.
\begin{lem}\label{lem2701}
Let $u\in C^\infty(M)$ be a smooth function on $M$. Let $h \in(0,\h)$ and  $s \in(0,h)$, $y\in\Gamma_s$ and  $\mathrm{H}(s)$ be the mean curvature of the parallel hypersurface $\Gamma_s$. From the second equality in Proposition \ref{prop2101}, at each point in $\Gamma_s$
\begin{equation*}
\mathrm{H}(s)=\frac{1}{n-1}\tr( \hess d_s)=\frac{1}{n-1}\sum_{i=1}^{n-1}\kappa_i
\end{equation*}
 and the following inequality holds:
\begin{equation}\label{onyep}
\left[(h-s)(n-1)\mathrm{H}(s)-1\right]|\grad u|^2\leqslant |\grad u|^2\div(\grad \eta)-2 \hess \eta(\grad u,\grad u).
\end{equation}
\end{lem}
\begin{proof}
Let $p\in\Gamma_s$, and $x=\grad u(p)$. Choose an orthonormal  frame such that  $\hess\eta(p)$  is diagonal and let $A = diag(\rho_1,\ldots,\rho_n)$ is the diagonal matrix representing the Hessian of $\eta$ at $p$. From the previous lemma, one has 
$$\rho_n|x|^2\geqslant A x\cdot x.$$
 Therefore,
$$\left(\sum_{i=1}^n\rho_i-2\rho_n\right)|x|^2\leqslant|x|^2\tr(A)-2 A x\cdot x.$$
From Lemma \ref{lem11022020a}, we have $\sum_{i=1}^n\rho_i(p)-2\rho_n=\sum_{i=1}^{n-1}\rho_i(p)-\rho_n=\sum_{i=1}^{n-1}(h-s)\kappa_i(p)-1=(n-1)(h-s)H(p)-1$ and the result follows.
\qedhere
\end{proof}
\begin{lem}\label{Lem2701}
Assumptions are the same as in Lemma \ref{lem2701}. Then the following inequality holds at each point in $\Gamma_s$.
\begin{equation}
-\left[(n-1)h \left( \sqrt{|K_-|}+|\kappa_-|\right)+1\right]|\grad u|^2\leqslant |\grad u|^2\div(\grad \eta)-2 \hess \eta(\grad u,\grad u).
\end{equation}
\end{lem}
\begin{proof}
let $\mu:I\longrightarrow\R$ be the  solution of 
\begin{equation}
\begin{cases}
\mu'+\mu^2+K_-=0\\
\mu(0)=-\kappa_-,
\end{cases}
\end{equation}
then, by the mean curvature comparison in Theorem \ref{MeanCurvatureComp}, we have  $H\geqslant-\mu$. 

Now set $\tilde{K}=-|K_-|$ and $\mu_0(t):=-\tilde{K}\frac{sn_{\tilde{K}}}{cs_{\tilde{K}}}(t)+|\kappa_-|$. We notice that $\mu_0$ satisfies:
$$\mu_0'=-\tilde{K}\frac{1}{cs_{\tilde{K}}^2},$$
$$\mu_0^2=\tilde{K}^2\frac{sn_{\tilde{K}}^2}{cs_{\tilde{K}}^2}-2\tilde{K}|\kappa_-|\frac{sn_{\tilde{K}}}{cs_{\tilde{K}}}+|\kappa_-|^2.$$
Then 
\begin{align*}
\mu_0'+\mu_0^2&=-\tilde{K}\left(1+2|\kappa_-|\frac{sn_{\tilde{K}}}{cs_{\tilde{K}}}\right)+|\kappa_-|^2\\
&\geqslant -\tilde{K}+|\kappa_-|^2\geqslant -\tilde{K}\geqslant-K_-,
\end{align*}
since 
$$ 0\leqslant-\tilde{K}\frac{sn_{\tilde{K}}(t)}{cs_{\tilde{K}}(t)}=
\begin{cases}
\sqrt{-\tilde{K}}\tanh(\sqrt{-\tilde{K}}t)\quad &\tilde{K}\neq 0,\\
0\quad &\tilde{K}\neq 0.
\end{cases}
$$
In addition, 
$$\mu_0(0)=|\kappa_-|\geqslant -\kappa_-.$$
Applying again Proposition \ref{riccacomp}, we get $\mu\leqslant\mu_0$. Hence we have $H\geqslant-\mu_0\geqslant -(\sqrt{|K_-|}+|\kappa_-|) $
 and the result follows from replacing the affected values in \eqref{onyep}.
\qedhere
\end{proof}
\begin{lem}
Let $u\in C^\infty(M)$ be a harmonic function on $M$. Let $\overline{B}:=2\left((n-1) \big(\sqrt{|K_-|}+|\kappa_-|\big)+\frac{1}{\h}\right) $ then the following inequality holds.
\begin{equation*}
 R_\beta(u)
 	\leqslant\left[\frac{1}{\sqrt{\beta}}+\sqrt{\overline{B}+ \beta\int_\Gamma  |\grad_\Gamma u|^2 \mathrm{d}_\Gamma }  \right]^2.
\end{equation*} 
\end{lem}
\begin{proof}
 Set  \begin{equation}\label{defb13022020}
 B:=(n-1) \big(\sqrt{|K_-|}+|\kappa_-|\big),
 \end{equation}
 by Lemma \ref{Lem2701} we have for every $h\in(0,\h)$
\begin{equation}\label{eq20022020a}
      -(Bh+1)\int_{M_h}{ |\grad u|^2 \mathrm{d}_M}\leqslant \int_{M_h}{ |\grad u|^2\div(\grad \eta)-2 \hess \eta(\grad u,\grad u) \mathrm{d}_M}.
\end{equation}
Therefore, applying \eqref{eq20022020a} and Proposition \ref{PohWentzel} with $h=\frac{\h}{2}$, we get
$$  \int_{M}{ |\grad u|^2\mathrm{d}_M}+ \beta \int_\Gamma  (\partial_\n u)^2 \mathrm{d}_\Gamma-2(B+\frac{1}{\h}) \beta \int_{M_\frac{\h}{2}}{ |\grad u|^2 \mathrm{d}_M}\leqslant R_\beta(u).$$
Set $\tilde{B}:=2(B+\frac{1}{\h}) \beta-1$, then one has
\begin{align*}
R_\beta(u)
	&\geqslant \beta  \left(\int_M{ |\grad u|^2 \mathrm{d}_M}\right)^2-\tilde{B} \int_M{ |\grad u|^2 \mathrm{d}_M},\\
	&=    \left(\sqrt{\beta}\int_M{ |\grad u|	^2 \mathrm{d}_M}\right)^2-2\left(\sqrt{\beta} \int_M{ |\grad u|	^2 \mathrm{d}_M}\right)\left(\frac{\tilde{B}}{2\sqrt{\beta}}\right).
\end{align*}
Hence,
\begin{align*}
  \left[\sqrt{\beta}\int_M{ |\grad u|	^2 \mathrm{d}_M}-\frac{\tilde{B}}{2\sqrt{\beta}}\right]^2&\leqslant R_\beta(u)+\left(\frac{\tilde{B}}{2\sqrt{\beta}}\right)^2,\\
  \left|\sqrt{\beta}\int_M{ |\grad u| \mathrm{d}_M}-\frac{\tilde{B}}{2\sqrt{\beta}}\right|&\leqslant \sqrt{R_\beta(u)}+\frac{|\tilde{B}|}{2\sqrt{\beta}},
\end{align*}
meaning that
\begin{align}
  \sqrt{\beta}\left( R_\beta(u)-\beta\int_\Gamma  |\grad_\Gamma u|^2 \mathrm{d}_\Gamma \right)-\frac{\tilde{B}}{2\sqrt{\beta}}&\leqslant \sqrt{R_\beta(u)}+\frac{|\tilde{B}|}{2\sqrt{\beta}},\nonumber\\
  \sqrt{\beta}  R_\beta(u)-\sqrt{ R_\beta(u)}-\frac{|\tilde{B}|}{\sqrt{\beta}}-\beta\sqrt{\beta}\int_\Gamma  |\grad_\Gamma u|^2 \mathrm{d}_\Gamma&\leqslant 0.\label{eq2}
\end{align}
Solving \eqref{eq2} with unknown $\sqrt{R_\beta(u)} $, we get 
\begin{align*}
 R_\beta(u)
 	&\leqslant\frac{1}{4\beta}\left[ 1+\sqrt{1+4|\tilde{B}|+4\beta^2\int_\Gamma  |\grad_\Gamma u|^2 \mathrm{d}_\Gamma }  \right]^2\\
 	&\leqslant\left[\frac{1}{\sqrt{\beta}}+\sqrt{\frac{|\tilde{B}|}{\beta}+ \beta\int_\Gamma  |\grad_\Gamma u|^2 \mathrm{d}_\Gamma }  \right]^2.
\end{align*} 
The result follows since $|\tilde{B}|\leqslant \beta\overline{B}$.
\qedhere
\end{proof}
\begin{proof}[Proof of Theorem \ref{mainth1}]
Take $\{\varphi_k\}_{k=0}^\infty$ an orthonormal basis $L^2(\Gamma)$ consisting of eigenfunctions of $\Delta_\Gamma$, such that $\Delta_\Gamma\varphi_k=\eta_k\varphi_k$ for $k\leqslant 0$.
Let $k$ be fixed and for each $j=0,1,\ldots,k$, $\phi_j$ be the harmonic extensions  of $\varphi_j$, i.e.
\begin{equation}\label{D}
\begin{cases}
\Delta \phi_j=0\quad\text{ in } \Omega,\\
\phi_j |_\Gamma= \varphi_j \quad \text{ on } \Gamma.
\end{cases}
\end{equation}
 Then $\phi_i\in \mathfrak{W}$, and 
\begin{align*}
\lambda_{W,k}^{\beta} &\leqslant \underset {0\leqslant j \leqslant k} {\max} R_\beta(\phi_j)\\
		&\leqslant\underset {0\leqslant j \leqslant k} {\max}\left[\frac{1}{\sqrt{\beta}}+\sqrt{\overline{B}+ \beta\int_\Gamma  |\grad_\Gamma \phi_j|^2 \mathrm{d}_\Gamma }  \right]^2\\
		&\leqslant\underset {0\leqslant j \leqslant k} {\max}\left[\frac{1}{\sqrt{\beta}}+\sqrt{\overline{B}+ \beta\eta_j}  \right]^2\\
		&\leqslant \left[\frac{1}{\sqrt{\beta}}+\sqrt{\overline{B}+ \beta\eta_k}  \right]^2.
\end{align*}
\qedhere
\end{proof}

\subsubsection{Second comparison}
We give here a proof of Theorem \ref{mainth2} using the following comparison.
\begin{thm}
Let $M\in \mathfrak{M}^n(K_-,K_+,\kappa_-,\kappa_+)$ and $\overline{h}\in\R_{>0}$ be the rolling radius of $M$. Let $a:[0,a_+)\longrightarrow\R$ and $b:[0,b_+)\longrightarrow\R$ be the maximal solutions of 
$$
\begin{cases}
a'+a^2+K_-=0\\
a(0)=-\kappa_-
\end{cases}
\text{and}\quad
\begin{cases}
b'+b^2+K_+=0\\
b(0)=-\kappa_+,
\end{cases}
$$
respectively.
 Then we have $b_+\leqslant\roll(M)\leqslant a_+$ and 
 \begin{equation}\label{eq12022020}
 \begin{cases}
-a(t)\leqslant \kappa_i(t)\quad &\forall~ t\in [0, \overline{h})\\
\kappa_i(t)\leqslant-b(t), &\forall~ t\in [0,  b_+).
\end{cases}
 \end{equation}
\end{thm}
\begin{proof}
We know from equality \eqref{RiccatiS} that the shape operator with eigenvalues $\{-\kappa_i,~i=1,\ldots, n-1\}$, in $M_{\overline{h}}$, satisfies 
\begin{equation*}
 S'+S^2+R=0.
 \end{equation*}
Applying Theorem \ref{RiccatiComparison} with $R_1 =R$, $R_2 =K_- I$,
$a_1(0) = S(0)$ and $a_2(0) = -\kappa_-$, we  get 
\begin{equation*}
\begin{cases}
S(t)\leqslant a(t)I\quad \text{for all } t\in[0,\overline{h})\\ \text{ and }\\
 \roll(M)\leqslant a_+.
 \end{cases}
\end{equation*} 
Hence, the principal curvatures of the level hypersurfaces in $M_{\overline{h}}$ satisfy 
$$-a(t)\leqslant \kappa_i(t).$$
 For the second inequality, we apply Theorem \ref{RiccatiComparison} with $R_1 =K_+$, $R_2 =R I$,
$a_1(0) = -\kappa_+$ and $a_2(0) = S(0)$, we  get similarly
\begin{equation*}
b(t)I\leqslant S(t) \quad \text{for all } t\in[0,b_+).
\end{equation*}
Hence, each principal curvature satisfies $\kappa_i(t)\leqslant-b(t)$.
\end{proof}

\begin{lem}\label{lm19022020a}
Let $u\in C^\infty(M) $ be a harmonic function. Let $h\in(0,\h)$ and  $s \in[0,\overline{h})$ then the following inequality holds at each point in $\Gamma_s$.
\begin{align}\label{onyep12}
 |\grad u|^2\div(\grad \eta)&-2 \hess \eta(\grad u,\grad u)\nonumber\\
 &\leqslant \left(1+\sum_{i=2}^{n-1}\rho_i-\rho_{1} \right)|\grad u|^2.
\end{align}
\end{lem}
\begin{proof}
The proof is similar to that of \eqref{onyep}.
If $p\in\Gamma_s$, and $x=\grad u(p)$, taking an orthonormal  frame such that  $\hess\eta(p)$  is diagonal, one has 
$$\rho_1|x|^2\leqslant A x\cdot x,$$
where $A = diag(\rho_1,\ldots,\rho_n)$ is the diagonal matrix representing the Hessian of $\eta$ at $p$. Therefore
$$\left(\sum_{i=1}^n\rho_i-2\rho_1\right)|x|^2\geqslant|x|^2\tr(A)-2 A x\cdot x.$$
The result follows from Lemma \ref{lem11022020a}, since $\sum_{i=1}^n\rho_i(p)-2\rho_1\geqslant 1+\sum_{i=2}^{n-1}\rho_i(p)-\rho_1.$
\qedhere
\qedhere
\end{proof}
\begin{lem}\label{lm20022020b}
Assumptions are the same as in Lemma \ref{lem11022020a}. For every $h\in(0, \h)$, we have
\begin{align*}
 \int_{M_{h}} |\grad u|^2\div(\grad \eta)&-2 \hess \eta(\grad u,\grad u) \mathrm{d}_M\\
 &\leqslant \left(h\frac{B}{n-1}+(n-1)\right)\int_M{ |\grad u|^2 \mathrm{d}_M}.
\end{align*}
The constant $\h$ is defined in \eqref{eq160202020b} and $B$ is the same as in \eqref{defb13022020}.
\end{lem}
\begin{proof}
From equations \eqref{onyep12}, \eqref{eq12022020} and Lemma \ref{lem11022020a}, for every $s\in [0, h)$, we have
\begin{align}
 |\grad u|^2\div(\grad \eta)&-2 \hess \eta(\grad u,\grad u)\nonumber\\
 &\leqslant \left(1+\sum_{i=2}^{n-1}\rho_i-\rho_{1} \right)|\grad u|^2\nonumber\\
 &\leqslant \Big(1+(n-2)+(h-s)a(h-s) \Big)|\grad u|^2.\label{eq13022020}
\end{align}
 
We set $\tilde{K}=-|K_-|$ and $\mu_0(t):=-\tilde{K}\frac{sn_{\tilde{K}}}{cs_{\tilde{K}}}(t)+|\kappa_-|$. We notice that $\mu_0$ satisfies:
$$\mu_0'=-\tilde{K}\frac{1}{cs_{\tilde{K}}^2},$$
$$\mu_0^2=\tilde{K}^2\frac{sn_{\tilde{K}}^2}{cs_{\tilde{K}}^2}-2\tilde{K}|\kappa_-|\frac{sn_{\tilde{K}}}{cs_{\tilde{K}}}+|\kappa_-|^2.$$
Then 
\begin{align*}
\mu_0'+\mu_0^2&=-\tilde{K}\left(1+2|\kappa_-|\frac{sn_{\tilde{K}}}{cs_{\tilde{K}}}\right)+|\kappa_-|^2\\
&\geqslant -\tilde{K}+|\kappa_-|^2\geqslant -\tilde{K}\geqslant-K_-,
\end{align*}
since 
$$ 0\leqslant-\tilde{K}\frac{sn_{\tilde{K}}(t)}{cs_{\tilde{K}}(t)}=
\begin{cases}
\sqrt{-\tilde{K}}\tanh(\sqrt{-\tilde{K}}t)\quad &\tilde{K}\neq 0,\\
0\quad &\tilde{K}= 0.
\end{cases}
$$
In addition, 
$$\mu_0(0)=|\kappa_-|\geqslant -\kappa_-.$$
Applying again Proposition \ref{riccacomp}, we get $a(t)\leqslant\mu_0(t)$. Hence we have $a(t)\leqslant \sqrt{|K_-|}+|\kappa_-| $ since $\tanh(x)\leqslant 1$ for every $x\in\R$. Replacing in 
inequality \eqref{eq13022020}, we get 
\begin{align*}
|\grad u|^2\div(\grad \eta)&-2 \hess \eta(\grad u,\grad u) \\
 &\leqslant \left[(n-1)+ (h-s)\left(\sqrt{|K_-|}+|\kappa_-|\right)\right]|\grad u|^2\\
 &\leqslant \left[(n-1)+ h\frac{B}{n-1}\right]|\grad u|^2.
\end{align*}
 
\end{proof}
\begin{proof}[Proof of Theorem \ref{mainth2}]
 We set $\overline{A}:=2\left(\frac{B}{n-1}+\frac{n-1}{\h}\right)$, then applying Proposition \ref{PohWentzel} and Lemma \ref{lm20022020b} with $h=\frac{\h}{2}$, one has
\begin{align*}
R_\beta(u)&\leqslant (1+\beta\overline{A})\int_{M}{ |\grad u|^2\mathrm{d}_M}+ \beta \int_\Gamma  (\partial_\n u)^2 \mathrm{d}_\Gamma\\
&\leqslant (1+\beta\overline{A})\left[\int_\Gamma  (\partial_\n u)^2 \mathrm{d}_\Gamma\right]^{\frac{1}{2}}+ \beta \int_\Gamma  (\partial_\n u)^2 \mathrm{d}_\Gamma
\end{align*}
Let $(\Psi_i)_{i\in\N}$ be a complete set of eigenfunctions corresponding to the Steklov eigenvalues $\lambda^S_i$ of $M$ forming an orthonormal basis of $\mathfrak{W}_0$. Let $\spn\{\Psi_i, i=1,\ldots,k\}$ be the trial space $V$. Every $u\in V$ such that $\int_\Gamma  u^2 \mathrm{d}_\Gamma=1$ can be written  as $u=\sum_{i=1}^k c_i\Psi_i$ with $\sum_{i=1}^k c_i^2=1$.
\begin{align*}
 \lambda_{W,k}^{\beta}&\leqslant (1+\beta\overline{A})\left[\int_\Gamma  (\partial_\n u)^2 \mathrm{d}_\Gamma\right]^{\frac{1}{2}}+ \beta \int_\Gamma  (\partial_\n u)^2 \mathrm{d}_\Gamma\\
 &\leqslant (1+\beta\overline{A})\left[\int_\Gamma  (\sum_{i=1}^k c_i\partial_\n\Psi_i )^2 \mathrm{d}_\Gamma\right]^{\frac{1}{2}}+ \beta \int_\Gamma  (\sum_{i=1}^k c_i\partial_\n\Psi_i)^2 \mathrm{d}_\Gamma\\
  &\leqslant (1+\beta\overline{A})\lambda^S_k\left[\int_\Gamma  u^2 \mathrm{d}_\Gamma\right]^{\frac{1}{2}}+ (\lambda^S_k)^2\beta \int_\Gamma u^2 \mathrm{d}_\Gamma\\
    &= (1+\beta\overline{A})\lambda^S_k+ \beta(\lambda^S_k)^2.
\end{align*}
\qedhere
\end{proof}
\section{Estimates based on Ricci curvature with Reilly identity}
Let $(M, g)$ be an $n$-dimensional compact connected Riemannian manifold with non-empty boundary $\Gamma$. We suppose that the Ricci curvature of $M$ and the principle curvatures of $\Gamma$ are  bounded from below. Then we get a quantitative comparison between the Wentzel-eigenvalues and the eigenvalues of the Laplace-Beltrami operator in $\Gamma$. We use the following Reilly's formula which has many interesting applications.
%

\begin{thm}[Reilly, 1977]
Given a smooth function $f$ on $M$, we denote $z = f|_\Gamma$ and $v = \partial_\n f$. Then,

\begin{equation}
\label{Reilly}
 \int_M ( \Delta f)^2-|\hess f|^2- \int_M \mathrm{Ric}( \nabla f,  \nabla f)=\int_\Gamma H v^2- 2v \Delta_\Gamma z +\Pi(\nabla_{\Gamma} z,\nabla_{\Gamma} z).
\end{equation}
\end{thm}
We refer the reader to \cite[(14)]{ReillyRobert} for further details.
%
\begin{proof}[Proof of Theorem \ref{AvecReilly}]
As is well known, $L^2(\Gamma)$ has an orthonormal basis made of eigenfunctions of $\Delta_\Gamma$ that we will denote by $\{\varphi_i\}_{k=0}^\infty$, such that, $\varphi_k$ is associated to $\eta_k$, that is $\Delta_\Gamma \varphi_k=\eta_k\varphi_k$, for all $k\geqslant 0$.\\
Let $k$ be fixed, consider the harmonic extensions $\phi_j \stackrel{\scriptscriptstyle\text{def}}= \wedge \varphi_j\in H^1(M)$:

\begin{equation}\label{D2}
\begin{cases}
\Delta \phi_j=0\quad\text{ in } \Omega,\\
\phi_j |_\Gamma= \varphi_j \quad \text{ on } \Gamma,
\end{cases}
\end{equation}
for $j = 1,\ldots,k$. Let  $V  \stackrel{\scriptscriptstyle\text{def}}= \spn\{\phi_0,\phi_1,\ldots,\phi_k\}$ be the space generated by $\{\phi_j\}_{j=0}^k$. 
We have $\varphi_j\in H^1(\Gamma)$, for all $j=0,\ldots,k$, since $\int_\Gamma \varphi_j^2=1$ and $\int_\Gamma |\nabla\varphi|^2=\eta_j$.
Thus, for any $\varphi_i$, we have a unique $\phi_i\in C^\infty(\overline{M})$ solving \eqref{D}, assuming each connected component of $M$ has non-empty boundary. We refer the interested reader to \cite[Sec~5]{Taylor}  for further information. Moreover, $\phi_i\in H^1(M)$ since $\varphi_i\in H^1(\Gamma)$. See \cite[p. 360, (1.39) and (1.40)]{Taylor}. Then accordingly, $\phi_i\in \mathfrak{W}$, for $i=0,\ldots,k$
and $V$ is a $k$-dimensional subspace of $\mathfrak{W}_\beta$.
Every function $\phi$ in  $V$  can be expressed as $\phi = \Sigma_{j=1}^k \alpha_j \phi_j$, then $(\phi,\phi)= \Sigma_{i=0}^k\alpha_i^2(\phi_i, \phi_i)$ and 
$(\mathrm{d}\phi,\mathrm{d} \phi) =\Sigma_{i=0}^k\alpha_i^2(\mathrm{d}\phi_i, \mathrm{d}\phi_i)$. 
Assume that $u\in\{ \phi_0,\ldots,\phi_k\} $ realises $R_\beta(u)=\underset{0\leqslant i \leqslant k} {\max}\int_\Om{|\nabla \phi_i|^2 \mathrm{d}_M}$ and
let $m= R_\beta(u)$, then
\begin{align*}
R_\beta(\phi)&=\frac{\Sigma_{i=0}^k\alpha_i^2\left( \int_M{|\nabla \phi_i|^2 \mathrm{d}_M}+\beta\int_\Gamma{|\nabla_\Gamma \phi_i|^2 \mathrm{d}_\Gamma} \right)}{\Sigma_{i=0}^k\alpha_i^2 \int_{\Gamma}{\phi_i^2 \mathrm{d}_\Gamma}}\\
& \leqslant \frac{\Sigma_{i=0}^k\alpha_i^2\left(m \int_{\Gamma}{\phi_i^2 \mathrm{d}_\Gamma}\right)}{\Sigma_{i=0}^k\alpha_i^2 \int_{\Gamma}{\phi_i^2 \mathrm{d}_\Gamma} }=m,
\end{align*}
meaning that $\{R_\beta(\phi);~ \phi\in V\}$ is bounded from above by $R_\beta(u)$.
Using the min-max principle \eqref{char} together with \eqref{Reilly} leads to
\begin{align*}
&\lambda_{W,k}^{\beta} \leqslant \underset {0\leqslant i \leqslant k} {\max} R_\beta(\phi_i)\leqslant \int_\Om{|\nabla u|^2 \mathrm{d}_M}+\beta\eta_k\\
& \leqslant\frac{1}{2(n-1)\kappa_-}\left[\left(K_-+2\eta_k \right)+\sqrt{(K_-+2\eta_k)^2-4\kappa_-(n-1)\eta_k} \right]+\beta\eta_k.
\end{align*}

Indeed, applying \eqref{rayleigh} to $u$ one has 
\begin{align*} 
-\int_M|\hess u|^2 & +Ric(\nabla u,\nabla u) \mathrm{d}_M\\
& = \int_\Gamma \left[H(\partial_\n u)^2-2|\nabla_\Gamma u|^2 \partial_\n u\right]\mathrm{d}_\Gamma+\Pi(\nabla_\Gamma u, \nabla_\Gamma u) \mathrm{d}_\Gamma.
\end{align*}
Since, by Cauchy-Schwarz inequality $(\Delta u)^2\leqslant n|\hess u|^2$, it follows that
$$-\frac{1}{n}  \int_M ( \Delta u)^2\mathrm{d}_M+ K_- \int_M |\nabla u|^2\mathrm{d}_M\geqslant\int_\Gamma H (\partial_\n u)^2- 2(\partial_\n u) \Delta_\Gamma u +\Pi(\nabla_{\Gamma} u,\nabla_{\Gamma} u)\mathrm{d}_\Gamma.
$$
Then,
\begin{align}\nonumber
K_-\int_M |\nabla u|^2\mathrm{d}_M & \geqslant\int_\Gamma H(\partial_\n u)\partial_\n u \mathrm{d}_\Gamma-2\eta_k\int_M |\nabla u|^2\mathrm{d}_M\nonumber\\
 &~~+ \int_\Gamma \Pi(\nabla_\Gamma u,\nabla_\Gamma u)\mathrm{d}_\Gamma\nonumber\\
 & \geqslant(n-1)\kappa_- \int_\Gamma (\partial_\n u)^2\mathrm{d}_\Gamma-2\eta_k\int_M |\nabla u|^2\mathrm{d}_M\nonumber\\
 &~~+\kappa_-\eta_k.\label{papillon1} 
\end{align}
Since $\int_M |\nabla u|^2 \mathrm{d}_M\leqslant\left(\int_\Gamma (\partial_\n u)^2 \mathrm{d}_\Gamma\right)^\frac{1}{2}$, \eqref{papillon1} can be transformed in two ways:
\begin{align}
\label{eqaigle1}
(n-1)\kappa_- \left(\int_M |\nabla u|^2 \mathrm{d}_M\right)^2&-\left(2\eta_k+K_-\right)\int_M |\nabla u|^2\mathrm{d}_M\nonumber\\
&+\eta_k\kappa_-\leqslant 0 
\end{align}
and
\begin{align}\label{eqaigle2}
K_-\left(\int_\Gamma (\partial_\n u)^2\mathrm{d}_\Gamma\right)^\frac{1}{2}  \geqslant &
(n-1)\kappa_- \int_\Gamma (\partial_\n u)^2\mathrm{d}_\Gamma-2\eta_k\left(\int_\Gamma (\partial_\n u)^2\mathrm{d}_\Gamma\right)^\frac{1}{2}\\
 &+\kappa_-\eta_k.\nonumber
\end{align}
 From \eqref{eqaigle2}, we have
\begin{align*}
&(n-1)\kappa_-\left[\int_\Gamma \left( \partial_\n u-\frac{K_-+2\eta_k}{2(n-1)\kappa_-}u\right)^2-\left[\frac{K_-+2}{2(n-1)\kappa_-} \right]^2u^2\mathrm{d}_\Gamma\right]\\
&~~+\int_\Gamma |\nabla_\Gamma u|^2\mathrm{d}_\Gamma\kappa_-  \leqslant 0.
\end{align*}
So, by a simple remarkable identity we get
\begin{align*}
 &(n-1)\kappa_- \int_\Gamma \left( \partial_\n u-\frac{K_-+2\eta_k}{2(n-1)\kappa_-}u\right)^2\mathrm{d}_\Gamma\\
 &~~-\kappa_-\left[\frac{(K_-+2\eta_k)^2}{4(n-1)\kappa_-^2}-\int_\Gamma |\nabla_\Gamma u|^2\mathrm{d}_\Gamma \right]  \leqslant 0.
\end{align*}
Ignoring all obvious non-negative terms in the last inequality, we have that   
$$\left[\frac{(K_-+2\eta_k)^2}{4(n-1)\kappa_-^2}-\eta_k \right]  \geqslant 0,$$ then
$$\mathit{Q}:=\left(K_-+2\eta_k\right)^2- 4(n-1) \kappa_-\eta_k\geqslant 0.$$
So solving the inequality \eqref{eqaigle1} for the unknown
$\int_M{|\nabla u|^2 \mathrm{d}_M}$, leads to 
\begin{multline}
\frac{1}{2(n-1)\kappa_-}\left[\left(K_-+2\eta_k \right)-\sqrt{\mathit{Q}} \right]\\
\leqslant\int_M|\nabla u|^2\mathrm{d}_M\\
\leqslant\frac{1}{2(n-1)\kappa_-}\left[\left(K_-+2\eta_k \right)+\sqrt{\mathit{Q}} \right].
\end{multline}

\end{proof}
\bibliographystyle{plain}
\providecommand{\bysame}{\leavevmode\hbox to3em{\hrulefill}\thinspace}
\providecommand{\noopsort}[1]{}
\providecommand{\mr}[1]{\href{http://www.ams.org/mathscinet-getitem?mr=#1}{MR~#1}}
\providecommand{\zbl}[1]{\href{http://www.zentralblatt-math.org/zmath/en/search/?q=an:#1}{Zbl~#1}}
\providecommand{\jfm}[1]{\href{http://www.emis.de/cgi-bin/JFM-item?#1}{JFM~#1}}
\providecommand{\arxiv}[1]{\href{http://www.arxiv.org/abs/#1}{arXiv~#1}}
\providecommand{\doi}[1]{\url{http://dx.doi.org/#1}}
\providecommand{\MR}{\relax\ifhmode\unskip\space\fi MR }
\providecommand{\MRhref}[2]{%
  \href{http://www.ams.org/mathscinet-getitem?mr=#1}{#2}
}
\providecommand{\href}[2]{#2}

\end{document}